\documentclass{amsart}
\usepackage{tikz}

\usepackage{amsmath,amsfonts,amsthm,amssymb,amscd, verbatim, graphicx,textcomp, hyperref}
\usepackage{lscape}
\usepackage{latexsym}
\usepackage{multirow}

\newtheorem{Thm}{Theorem}[section]
\newtheorem{Cor}[Thm]{Corollary}

\newtheorem{Qu}[Thm]{Question}
\newtheorem{Lem}[Thm]{Lemma}
\newtheorem{Rem}[Thm]{Remark}

\theoremstyle{definition}

\newtheorem{Ex}[Thm]{Example}

\theoremstyle{remark}

\numberwithin{equation}{section}

\newcommand{\Aut}{\operatorname{Aut}}

\newcommand{\Mor}{\operatorname{Mor}}
\newcommand{\stab}{\operatorname{stab}}
\newcommand{\Id}{\operatorname{Id}}

\renewcommand{\dim}{\operatorname{dim}}

\newcommand{\Stab}{\operatorname{Stab}}

\newcommand{\wt}{\operatorname{wt}}

\newcommand{\De}{\mathcal{D}}

\newcommand{\Sym}{\operatorname{Sym}}

\newcommand{\supp}{\operatorname{supp}}
\newcommand{\Alt}{\operatorname{Alt}}
\newcommand{\PSL}{\operatorname{PSL}}

\newcommand{\PG}{\operatorname{PG}}
\newcommand{\AG}{\operatorname{AG}}

\newcommand{\Sp}{\operatorname{Sp}}

\renewcommand{\Gamma}{\varGamma}
\renewcommand{\epsilon}{\varepsilon}

\renewcommand{\leq}{\leqslant}
\renewcommand{\geq}{\geqslant}

\newcommand{\I}{\mathcal{I} }
\newcommand{\B}{\mathcal{B} }

\newcommand{\ep}{\epsilon}

\newcommand{\Z}{\mathbb{Z} }

\renewcommand{\B}{\mathcal{B}}

\renewcommand{\L}{\mathcal{L}}

\newcommand{\E}{\mathcal{E}}

\renewcommand{\L}{\mathcal{L}}
\newcommand{\C}{\mathcal{C}}

\newcommand{\mF}{\mathbb{F}}

\begin{document}


\title{Conway's groupoid and its relatives}
 

\author{Nick Gill}
\address{Department of Mathematics, University of South Wales, Treforest, CF37 1DL, U.K.}
\email{nick.gill@southwales.ac.uk}
\author{Neil I. Gillespie}
\address{Heilbronn Institute for Mathematical Research, Department of Mathematics, University of Bristol, U.K.}
\email{neil.gillespie@bristol.ac.uk}
\author{Cheryl E. Praeger}
\address{Centre for the Mathematics of Symmetry and Computation, University of Western Australia, Australia\\
also affiliated with King Abdulaziz University Jeddah, Saudi Arabia}
\email{cheryl.praeger@uwa.edu.au}

\author{Jason Semeraro}
\address{Heilbronn Institute for Mathematical Research, Department of Mathematics, University of Bristol, U.K.}
\email{js13525@bristol.ac.uk}



\begin{abstract}
In 1997, John Conway constructed a $6$-fold transitive subset $M_{13}$ of permutations on a set of size $13$ for which
the subset fixing any given point was isomorphic to the Mathieu group $M_{12}$. The construction was via  a 
``moving-counter puzzle'' on the projective plane $\PG(2,3)$. We discuss consequences and generalisations 
of Conway's construction. In particular we explore how various designs and hypergraphs can be used instead of 
$\PG(2,3)$ to obtain interesting analogues of $M_{13}$; we refer to these analogues as {\it Conway groupoids}.  A number of 
open questions are presented.
\end{abstract}

\keywords{}

\subjclass[2010]{20B15, 20B25, 05B05}

\maketitle 

\section{The first Conway groupoid \texorpdfstring{$M_{13}$}{M13}}

In 1997, John Conway published a celebrated paper \cite{Co1} in which he constructed the sporadic simple group 
of Mathieu, $M_{12},$ via a ``moving-counter puzzle'' on the projective plane $\PG(2,3)$ of order $3$.  Conway 
noticed some new structural links between two permutation groups, namely  $M_{12}$, which acts 5-transitively on 
12 letters, and $\PSL_3(3)$, which acts 2-transitively on the 13 points of $\PG(2,3)$. These led him to his 
construction of $M_{13}$. In \cite[page 1]{Co1} he writes:

\begin{quote}
To be more precise, the point-stabilizer in $\PSL_3(3)$ is a group of structure $3^2:2\Sym(4)$ that permutes the 12 remaining points imprimitively in four blocks of 4, and there is an isomorphic subgroup of $M_{12}$ that permutes the 12 letters in precisely the same fashion. Again, the line-stabilizer in $\PSL_3(3)$ is a group of this same structure, that permutes the 9 points not on that line in a doubly transitive manner, while the stabilizer of a triple in $M_{12}$ is an isomorphic group that permutes the 9 letters not in that triple in just the same manner.

In the heady days when new simple groups were being discovered right and left, this common structure suggested that there should be a new group that contained both $M_{12}$ and $\PSL_3(3)$, various copies of which would intersect in the subgroups mentioned above.
\end{quote}

\noindent
The putative ``new group'' does not of course exist, but Conway's construction of $M_{12}$ using a certain 
puzzle on $\PG(2,3)$ did yield a natural definition of a \emph{subset} of permutations that contains both 
of these groups in the manner just described, and Conway called it $M_{13}$.
We discuss this puzzle briefly in Subsection~\ref{puzzle}, and then present some recently discovered 
analogues of $M_{12}$ and $M_{13}$ that can be obtained by variants of the puzzle, along with some of 
the geometry associated with these objects, especially focussing on connections to codes. Finally, 
we discuss attempts to classify such puzzles from a geometric and algebraic point of view.

\subsection{Conway's original puzzle}\label{puzzle}

We now present Conway's original puzzle in terms of permutations, rather than the ``counters'' used 
in \cite[Section 2]{Co1}. This description bears little resemblance to what we traditionally think of 
as a ``puzzle'', but contains all of the salient mathematics. Our notation, too, is different from that 
of Conway but prepares the way for what will come later.

We write $\Omega$ for the set of 13 points of $\PG(2,3)$. Each of the 13 lines of $\PG(2,3)$ is incident 
with exactly 4 points, and each pair of points is incident with exactly one line. We think of a line as 
simply a 4-subset of $\Omega$. 
Then, given any pair of distinct points $a,b\in \Omega$ we define the \emph{elementary move}, denoted 
$[a,b]$ to be the permutation $(a,b)(c,d)$ where $\{a,b,c,d\}$ is the unique line in $\PG(2,3)$ containing 
$a$ and $b$. For a point $a\in\Omega$, it is convenient to define the \emph{move} $[a,a]$ to be the 
identity permutation. Then, given a sequence of (not necessarily distinct)  points $a_1, a_2, \dots, 
a_\ell$ we define the \emph{move}
\begin{equation}\label{eq:move}
 [a_1,a_2,\dots, a_\ell] = [a_1,a_2]\cdot [a_2,a_3] \cdots [a_{\ell-1}, a_\ell].
\end{equation}
Note that we apply $[a_1,a_2]$ first, and then $[a_2,a_3]$, and so on, so this move maps $a$ to $b$.\footnote{The
terminology intentionally suggests a ``puzzle'' in which, for example, the move $[\infty,a,b]$ denotes ``moving''
a ``counter'' from $\infty$ first to $a$, using the elementary move $[\infty,a]$, and then moving it from $a$ to 
$b$, using the elementary move $[a,b]$, and so on.}  
Observe that all of these moves are elements of $\Sym(13)$. 

We now choose a point of $\PG(2,3)$, label it $\infty$, and consider two subsets of $\Sym(13)$:
\begin{align}\label{eq:LPi1}
 \pi_\infty(\PG(2,3)) &:= \{ [a_1,a_2,\dots, a_\ell] \mid 1<\ell \in \Z, a_1,\dots, a_\ell \in \Omega, a_1=a_\ell=\infty\}; \nonumber \\
 \L_\infty(\PG(2,3)) &:= \{ [a_1,a_2,\dots, a_\ell] \mid 1<\ell \in \Z, a_1,\dots, a_\ell \in \Omega, a_1=\infty\}.
\end{align}
It is easy to see that the set $\pi_\infty(\PG(2,3))$ (which we call the \emph{hole-stabilizer}) is, in fact, 
a subgroup of $\Sym(\Omega\setminus\{\infty\})\cong\Sym(12)$. Much less trivial is the spectacular fact due 
to Conway that $\pi_\infty(\PG(2,3))$ is isomorphic to $M_{12}$ (\cite[Sections 3 and 7]{Co1}).
The set $\L_\infty(\PG(2,3))$ is a subset of $\Sym(13)$ that contains $M_{12}$. It has size $13\cdot |M_{12}|$ 
and is equal to the product $\pi_\infty(\PG(2,3))\cdot \Aut(\PG(2,3))$. Since $\Aut(\PG(2,3))\cong \PSL(3,3)$, 
the set $M_{13}$ contains both $M_{12}$ and $\PSL(3,3)$ and, moreover, contains copies of these groups intersecting 
in exactly the manner that Conway proposed in the quote above.

\subsection{Some variants }\label{s: variants}

A number of variants of the $M_{13}$-puzzle are mentioned by Conway in his original article \cite{Co1}. Other 
variants were investigated further by Conway, Elkies and Martin \cite{CEM}, two of which relate to $\PG(2,3)$ and are particularly interesting:
\begin{enumerate}
 \item {\bf The signed game}. One defines moves on $\PG(2,3)$ as before, except that the definition of an 
elementary move also assigns a ``sign'' to each letter in the permutation. We write $[a,b]=(a,b)(\underline{c},
\underline{d})$ to denote that the letters $c$ and $d$ are given negative signs. The resulting hole-stabilizer, 
$\underline{\pi_\infty}(\PG(2,3)),$ may be regarded as a subgroup of the wreath product $\Z/2\Z\wr \Sym(12)$ 
and turns out to be isomorphic to $2M_{12}$, the double-cover of $M_{12}$, \cite[Theorem 3.5]{CEM}.

 \item{\bf The dualized game}. In this puzzle, the set $\Omega$ is the union of the point-set and the line-set 
of $\PG(2,3)$. We distinguish both a point $\infty$, and a line $\overline{\infty}$ such that $\infty$ and 
$\overline{\infty}$ are incident in $\PG(2,3)$. Since $\PG(2,3)$ is self-dual, one can define moves, as in 
the original game, for sequences of points as well as sequences of lines. Once one has done this, one can 
define a move of the form
 \[
  [p_1, q_1, p_2, q_2, \cdots, p_\ell, q_\ell] = [p_1,\dots, p_\ell]\cdot [q_1,\dots, q_\ell]
 \]
where $p_1,\dots, p_\ell$ are points $q_1,\dots, q_\ell$ are lines and we require that $q_i$ is incident 
with $p_i$ for all $i=1,\dots, \ell$, and $q_i$ is incident with $p_{i+1}$ for all $i=1,\dots, \ell-1$. 
One can define analogously a hole-stabilizer $\pi_\infty^d(\PG(2,3))$, except that here we require $p_1=
p_\ell=\infty$ and $q_1=q_\ell=\overline{\infty}$. The group $\pi_\infty^d(\PG(2,3))$ is isomorphic to 
$M_{12}$ and its action on $\Omega$ splits into two orbits: the point-set and line-set. By interchanging the point-set and line-set appropriately one can obtain a concrete representation of the outer automorphism 
of $M_{12}$, \cite[Section 4]{CEM}.
\end{enumerate}

\subsection{Multiple transitivity}

The groups $M_{12}, \Sym(5), \Sym(6)$ and $\Alt(7)$ are the only finite permutation 
groups which are  sharply 5-transitive, that is, they are transitive on ordered $5$-tuples of distinct points 
and only the identity fixes such a $5$-tuple.  Moreover, if $G$ is a sharply $k$-transitive group with 
$k\geq 6$, then $G=\Sym(k), \Sym(k+1)$ or $Alt(k+2)$, (see \cite[Theorem 7.6A]{permutation}). 
The set $M_{13}$, however, having size $13\cdot |M_{12}|$, seemed a good candidate to be a ``sharply $6$-transitive 
subset of permutations'', and to clarify the meaning of this phrase, Conway, Elkies and Martin introduced the 
following notions in \cite[Section 5.1]{CEM}. Here $\mathcal{P}$ denotes the set of all ordered $6$-tuples of 
distinct points of $\PG(2,3)$.

\begin{itemize}
 \item A tuple ${\bf p}\in\mathcal{P}$ is a \emph{universal donor} if, for all 
 ${\bf q}\in\mathcal{P}$, there exists $g\in M_{13}$ such that ${\bf p}^g={\bf q}$.

 \item A tuple ${\bf q}\in\mathcal{P}$ is a \emph{universal recipient} if, for all ${\bf p}\in\mathcal{P}$, 
there exists $g\in M_{13}$ such that ${\bf p}^g={\bf q}$.
\end{itemize}
They proved the following result \cite[Theorems 5.2 and 5.3]{CEM} which gives full information regarding 
the sense in which $M_{13}$ is a sharply 6-transitive subset.

\begin{Thm}\label{t: 6t}
 \begin{enumerate}
  \item A tuple ${\bf p}=(p_1,\dots, p_6)\in\mathcal{P}$  
  is a universal donor if and only if $p_i=\infty$ for some $i$.
  \item A tuple ${\bf q}\in\mathcal{P}$ is a universal recipient if and only if ${\bf q}$ contains a line of $\PG(2,3)$. 
 \end{enumerate}
\end{Thm}
An alternative approach to the study of multiply-transitive sets of permutations was proposed and studied by Martin and Sagan \cite{MS}; earlier studies in the wider context of ``sharp subsets'' of permutations are discussed in \cite{Cam88}.
The problem of computing the transitivity of $M_{13}$ in the sense of Martin and Sagan was completed by Nakashima \cite{Nak}.
Sharp $k$-transitivity of subsets of permutations was also studied by Bonisoli and Quattrocchi \cite{BQ}. Their result 
is very strong although it applies only to so-called ``invertible'' sets (and $M_{13}$ is not invertible). 

\section{A more general setting for groupoids}\label{s: general}

For the rest of this paper we turn our attention to work inspired by Conway's construction of $M_{13}$, and which seeks to generalize it in various ways. 
This more general setting was first considered in \cite{GGNS}; it involves the notion of a \emph{$4$-hypergraph}, namely a pair $\De:=(\Omega,\B)$, where $\Omega$ is a finite 
set of size $n$, and $\B$ is a finite multiset of subsets of $\Omega$ (called lines), each of size $4$. Observe 
that $\PG(2,3)$ is a $4$-hypergraph on a set of size $13$. A pair $a, b$ of (not necessarily distinct) points is called 
\emph{collinear} if $a, b$ are contained in some line of $\De$, and $\De$ is said to be \emph{connected} if, for all 
$a, b\in\Omega$, there exists a finite sequence $a_0=a, a_1,\dots,a_k=b$ of points from $\Omega$ such that
each pair $a_{i-1}, a_i$ is collinear.

Consider an arbitrary connected $4$-hypergraph $\De=(\Omega, \B)$. For a pair of distinct collinear points $a,b\in \Omega$ 
we define the \emph{elementary move}, denoted $[a,b]$, to be the permutation 
\[
(a,b)(c_1,d_1)(c_2,d_2)\cdots(c_\lambda,d_\lambda)
\]
where $\{a,b,c_i,d_i\}$ (for $i=1,\dots, \lambda$) are the lines of $\De$  containing $a$ and $b$. The
value of $\lambda$ in general depends on $a$ and $b$. To ensure that each elementary move is well defined, 
the $4$-hypergraph $\De$ is required to be \emph{pliable}, that is, whenever two lines have at least three 
points in common, the two lines contain exactly the same points. 

The rest of the set-up proceeds {\it \'a la} the analysis of $M_{13}$ given at the start of \S\ref{puzzle}: We define the move $[a,a]$ to be the identity permutation, for each $a$,  and for a sequence $a_1, a_2,\dots,a_\ell$ such 
that each pair $a_{i-1}, a_i$ is collinear, we define the move  $[a_1,\dots, a_n]$ exactly as in \eqref{eq:move}. Finally, we distinguish a point of $\De$ which we call $\infty$, and we define the hole-stabilizer $\pi_\infty(\De)$ and the set $\L_\infty(\De)$ as in \eqref{eq:LPi1}. 
The set $\pi_\infty(\De)$ is again a subgroup of $\Sym(\Omega\setminus\{\infty\})$, and the subset $\L_\infty(\De)$ 
of $\Sym(\Omega)$ is an analogue of Conway's $M_{13}$. 

In \cite{Co1}, Conway recognised that $M_{13}$ could be endowed with the structure of a \emph{groupoid} (that is,
\ a small category in which all morphisms are isomorphisms). The set $M_{13}$ is sometimes referred to as the 
\emph{Mathieu groupoid}. We define an analogue of this notion in this more general setting, and explain the 
connection between the set and the category.

For a pliable, connected $4$-hypergraph $\De=(\Omega, \B)$, the \textit{Conway groupoid} $\C(\De)$ is the small category 
whose object set is $\Omega$, such that, for $a,b\in \Omega,$ the set $\Mor(a,b)$ of morphisms from $a$ to $b$ is precisely
$$
\Mor(a,b) :=\{[a,a_1,\ldots,a_{k-1},b] \mid a_{i-1}, a_i\in\Omega \mbox{ for } 1 \leq i \leq k-1\} .
$$ 
Since $\De$ is connected, there exists a finite sequence $\infty=b_0, b_1,\ldots,b_{\ell} = a$ such that 
each pair $b_{i-1}, b_i$ is collinear. Hence $\rho := [\infty, b_{1},\dots,b_{\ell-1}, a]  \in \L_\infty(\De)$.
Moreover, for each $b\in\Omega$ and each $\sigma = [a,a_1,\ldots,a_{k-1},b] \in \Mor(a,b)$, we also have
$\tau := [\infty,   b_{1},\dots,b_{\ell-1}, a, a_1,\ldots,a_{k-1},b] \in \L_\infty(\De)$, and  
$\sigma= \rho^{-1} \cdot \tau$. In particular, the category $\C(\De)$ is completely determined 
by the set $\L_\infty(\De)$. Thus, just as the term \emph{Mathieu groupoid} is applied in the literature 
to both $\C(\PG(2,3))$ and $\L_\infty(\PG(2,3))$, so also the term \emph{Conway groupoid} is used for both 
$\C(\De)$ and $\L_\infty(\De)$ (although we tend to focus on the latter).
The following result, which follows from \cite[Lemma 3.1]{GGNS}, is relevant.

\begin{Lem}\label{l: iso}\cite{GGNS}
Let $\De$ be a pliable $4$-hypergraph for which each pair of points is collinear.
Let $\infty_1, \infty_2$ be points of $\De$. Then $\pi_{\infty_1}(\De)\cong \pi_{\infty_2}(\De)$ (as permutation groups).
\end{Lem}

This lemma can be strengthened so that we obtain the same conclusion, supposing only that $\De$ is a 
pliable, connected $4$-hypergraph. The lemma allows us to talk about ``the'' hole stabilizer of such a 
hypergraph without having to specify the base point $\infty$. A similar statement also holds for the sets $\L_{\infty_1}(\De)$ and $\L_{\infty_2}(\De)$, allowing us to talk about ``the'' hole stabilizer of a pliable, connected $4$-hypergraph.

\subsection{Some examples}\label{s: examples}

For this section, we need some definitions: for positive integers, $n,k,\lambda$, a \emph{$2-(n,k,\lambda)$-design} 
$(\Omega, \B)$ consists of a set $\Omega$ of ``points'' of size $n$, and a multiset $\B$ of $k$-element subsets of $\Omega$ 
(called ``lines'') such that any $2$-subset of $\Omega$ lies in exactly $\lambda$ lines.   
The design $(\Omega, \B)$ is called \emph{simple} if there are no repeated lines (that is,\ $\B$ is a set, 
rather than a multiset). If $k=4$ and $(\Omega, \B)$ is a simple $2-(n,4,\lambda)$-design, then $(\Omega, \B)$ 
is a connected $4$-hypergraph.  Further, if in addition $(\Omega, \B)$ is pliable, that is, if distinct lines 
intersect in a set of size at most $2$, then $(\Omega, \B)$ is called \emph{supersimple}. 

The search for examples of interesting new Conway groupoids, which we report on, has focussed almost exclusively 
on the situation where the $4$-hypergraph $\De=(\Omega, \B)$ is a supersimple $2-(n,4,\lambda)$ design. 
In particular, Lemma~\ref{l: iso} applies and the isomorphism class of the hole-stabilizer is, up to permutation isomorphism,  independent of the choice of the point $\infty$.

Let us first consider a somewhat degenerate case: it turns out that $\pi_\infty(\De)=\Alt(\Omega\setminus\{\infty\})$ if and only if $\L_\infty(\De)=\Alt(\Omega)$, and  $\pi_\infty(\De)=\Sym(\Omega\setminus\{\infty\})$ if and only if $\L_\infty(\De)=\Sym(\Omega)$. In 
these cases the puzzle-construction sheds no new light on the groups in question, and can be safely ignored. 
It turns out that for very many of the supersimple $2-(n,k,\lambda)$ designs $\De$ examined, the corresponding 
Conway groupoid is of this type. Indeed it turns out that if $n$ is sufficiently large relative to $\lambda$, 
then $\L_\infty(\De)$ always contains $\Alt(\Omega)$.

\begin{Lem}\label{lem:nlambda}\cite[Theorem E(3)]{GGS}\quad
If $\De=(\Omega, \B)$ is a supersimple $2-(n,4,\lambda)$ design with $n>144 \lambda^2 + 120 \lambda +26$, 
then $\L_\infty(\De)\supseteq\Alt(\Omega)$.
\end{Lem} 

We shall have more to say about the relationship between $n$ and $\lambda$ in Theorem~\ref{t: n and lambda}. However 
this crude bound is sufficient to show, for example, that the Conway groupoids for the point-line designs of 
projective spaces $\PG(r,3)$ and affine spaces $\AG(r,4)$ contain $\Alt(\Omega)$ whenever $r\geq5$.
The first infinite family of examples without this property was studied in \cite{GGNS}.

\begin{Ex}\label{ex: boolean}
 The \textit{Boolean quadruple system of order $2^m$}, where $m\geq2$, is the design 
$\De^b=(\Omega^b,\B^b)$ such that $\Omega^b$ is identified with the set of vectors in $\mathbb{F}_2^m$, and 
 \[
 \B^b:=\{\{v_1,v_2,v_3,v_4\} \mid v_i \in \Omega^b \mbox{ and } \sum_{i=1}^4 v_i = \textbf{0}\}.
 \]
Equivalently, we can define 
\[
\B^b=\{v+W \mid v\in \Omega^b, W\leq \mathbb{F}_2^m, \dim(W)=2\};
\]
that is, $\B^b$ is the set of all affine planes of $\Omega^b$.
It is easy to see that $\De$ is both a $3$-$(2^m,4,1)$ Steiner quadruple system and a 
supersimple $2$-$(2^m,4,2^{m-1}-1)$ design.
It turns out that $\pi_\infty(\De^b)$ is trivial. In addition, the Conway groupoid $\L_\infty(\De^b)$ 
is equal to the group of translations $E_{2^m}$ acting transitively on $\Omega^b$ \cite[Theorem B, Section 5]{GGNS}.
\end{Ex}

The approach in the literature to finding new examples from supersimple designs has tended to be 
organised in terms of the behaviour of the hole stabilizer (as a subgroup of 
$\Sym(\Omega\setminus\{\infty\})$). For a Boolean quadruple system $\De^b$, the trivial group 
$\pi_\infty(\De^b)$ is clearly intransitive on the set $\Omega\setminus\{\infty\}$. However, 
examples with intransitive hole-stabilizers seem quite rare. We only know three other examples. 
They are given in \cite[Table 1]{GGNS}, and have parameters: 
 \begin{enumerate}
  \item $(n,\lambda)=(16,6)$ and $\pi_\infty(\De)\cong \Sym(3)\times \Sym(3) \times \Sym(3) \times \Sym(3) \times \Sym(3)$;
  \item $(n,\lambda)=(17,6)$ and $\pi_\infty(\De)\cong \Sym(8)\times \Sym(8)$;
  	\item $(n,\lambda)=(49,18)$ and $\pi_\infty(\De)\cong \Sym(24)\times \Sym(24)$.
 \end{enumerate}


\begin{Qu}\label{q: intrans}
Apart from Boolean quadruple systems, are there infinitely many supersimple designs for which the hole-stabilizers are intransitive?
\end{Qu}

We next turn our attention to the situation where the hole-stabilizers are transitive
on  $\Omega\setminus\{\infty\}$. Examples for which this action is imprimitive also 
seem to be rare. Only one example appears in  \cite[Table 1]{GGNS}: it has parameters
 $(n,\lambda)=(9,3)$ and $\pi_\infty(\De)\cong \Alt(4)\wr C_2$.


\begin{Qu}\label{q: imprim}
 Are there more examples of supersimple designs $\De$ for which $\pi_\infty(\De)$ is transitive and imprimitive?
\end{Qu}

In \cite{GGS} two infinite families of designs are studied for which the hole-stabilizers 
are primitive. To describe them we need the following set-up: Let $m \geq 2$ and $V:=(\mathbb{F}_2)^{2m}$ 
be a vector space equipped with the standard basis. Define 
\begin{equation}\label{e:form}
e:=\begin{pmatrix}0_m&I_m\\0_m&0_m\end{pmatrix}, \qquad f:=\begin{pmatrix}0_m&I_m\\I_m&0_m\end{pmatrix}=e+e^T,
\end{equation}
where $I_m$ and $0_m$ represent the $m \times m$ identity and zero matrices respectively. We write elements of $V$ as 
row vectors and define $\varphi(u,v)$ as the alternating bilinear form  $\varphi(u,v):=ufv^T$. We also set $\theta(u):=ueu^T \in \mathbb{F}_2$, so that 
$$
\theta(u+v)+\theta(u)+\theta(v)=\varphi(u,v).
$$ 
(Note that the right hand side equals $u e v^T + v e u^T$ while the left hand side is $u(e+e^T)v^T$.) 
Finally, for each $v \in V$ define $\theta_v(u):=\theta(u)+\varphi(u,v)$, and note that $\theta_0=\theta$.

\begin{Ex}\label{ex: symplectic}
The \textit{Symplectic quadruple system of order $2^{2m}$}, where $m\geq2$, is the design
$\De^a=(\Omega^a,\B^a)$, where $\Omega^a:=V$ and 
$$
\B^a:=\left\{\{v_1,v_2,v_3,v_4\} \mid v_1,v_2,v_3,v_4 \in \Omega^a, \sum_{i=1}^4 v_i=\textbf{0}, \sum_{i=1}^4 \theta(v_i)=0 \right\}.
$$ 
By \cite[Theorem B]{GGS} for $m\geq3$ and \cite[Table 1]{GGNS} for $m=2$, $\L_\infty(\De^a) \cong 2^{2m}.\Sp_{2m}(2)$, while $\pi_\infty(\De^a) \cong \Sp_{2m}(2)$. Indeed, taking $\infty$ to be the zero vector in $V$, $\pi_\infty(\De^a)={\rm Isom(V,\varphi)}$, the isometry group 
 of the formed space $(V, \varphi)$.
 \end{Ex}

\begin{Ex}\label{ex: orthogonal}
The \textit{Quadratic quadruple systems of order $2^{2m}$}, where $m\geq3$, are the 
designs $\De^\ep=(\Omega^\ep,\B^\ep)$, for $\ep \in \mathbb{F}_2$, such that 
$\Omega^\ep:=\{\theta_v  \mid v \in V, \theta(v)=\ep\}$ and  
$$
\B^\ep:=\left\{\{\theta_{v_1},\theta_{v_2},\theta_{v_3},\theta_{v_4}\} \mid 
\theta_{v_1},\theta_{v_2},\theta_{v_3},\theta_{v_4} \in \Omega^\ep, \sum_{i=1}^4 v_i= \textbf{0} \right\}.
$$
By \cite[Theorem B]{GGS}, $\L_\infty(\De^\ep) \cong \Sp_{2m}(2)$, the 
isometry group of $\varphi$, while $\pi_\infty(\De^\ep) \cong \rm{O}_{2m}^{\ep'}(2)$, 
where $\ep'=\pm$ and $\ep=(1-\ep'1)/2$ (as an integer in $\{0,1\})$.
\end{Ex}


We remark that the set of lines in $\De^a$ coincides with the set of translates 
of the totally isotropic $2$-subspaces of $\Omega^a$. This alternative interpretation 
provides a link with Example~\ref{ex: boolean}. The designs in Example~\ref{ex: orthogonal} 
can be rephrased similarly (see \cite[\S6]{GGPS}). Note, too, that Example~\ref{ex: orthogonal} 
can be extended to include the case $m=2$, but only for $\ep = 0$. In this case, \cite[Table 1]{GGNS} 
asserts that $\L_\infty(\De^0) \cong \Sym(6)$ and $\pi_\infty(\De^0) \cong O^+_4(2)=\Sym(3)\wr\Sym(2)$.

The examples described thus far represent all those known for which $\De$ is a supersimple design. Note that in Examples~\ref{ex: boolean}, \ref{ex: symplectic} and \ref{ex: orthogonal}, the Conway groupoid $\L_\infty(\De)$ is always a group -- we shall have more to say on this phenomenon in \S\ref{s: two graphs}.

\section{Conway groupoids and codes}\label{s: codes}

In this section we consider certain codes that arise naturally from the supersimple designs 
initially used to define Conway groupoids. We use the following terminology from coding theory. 
A \emph{code} of length $m$ over an alphabet $Q$ of size $q$ is a subset of  
vertices of the Hamming graph $H(m,q)$, which is the graph $\Gamma$ with vertex set 
$V(\Gamma)$ consisting of all $m$-tuples with entries from $Q$, and such that two 
vertices are adjacent if they differ in 
precisely one entry. Consequently, the (Hamming) distance $d(\alpha,\beta)$ between two vertices 
$\alpha,\beta\in V(\Gamma)$ is equal to the number of entries in which they differ. If $Q$ is 
a finite field $\mF_q$, then we identify $V(\Gamma)$ with the space $\mF_q^m$ of $m$-dimensional
row vectors. In this case a code is called \emph{linear} if it is a subspace of $\mF_q^m$. 
We only consider linear codes in this section.

The \emph{support} of a vertex $\alpha=(\alpha_1,\ldots,\alpha_m)\in V(\Gamma)$ is the set 
$\supp(\alpha)=\{i\,|\,\alpha_i\neq 0\}$, and the \emph{weight} of $\alpha$ is
$\wt(\alpha)=|\supp(\alpha)|$.
Given a code $C$ in $H(m,q)$, the \emph{minimum distance} of $C$ is the minimum of $d(\alpha,\beta)$ 
for distinct codewords $\alpha,\beta\in C$, and for
a vertex $\beta\in V(\Gamma)$, the distance from $\beta$ to $C$ is defined as 
$$
d(\beta,C)=\min\{d(\beta,\alpha)\,|\,\alpha\in C\}.
$$ 
The \emph{covering radius $\rho$ of $C$} is 
the maximum of these distances:
$$
\rho=\max\{d(\beta,C)\,|\,\beta\in V(\Gamma)\}.
$$ 
For $i=0\ldots,\rho$ we let 
$$
C_i=\{\beta\,|\,d(\beta,C)=i\},
$$ 
so $C_0=C$, and 
we call the partition $\{C,C_1,\ldots,C_\rho\}$ of $V(\Gamma)$ the \emph{distance partition of $C$}.

A code is \emph{completely regular} if its distance partition is \emph{equitable}, 
that is, if the number of vertices in $C_j$ adjacent
to a vertex in $C_i$ depends only on $i,j$ and not on the choice of the vertex (for 
all $i,j$). Such codes have a high degree of combinatorial symmetry, and have been 
studied extensively (see, for example, \cite{distreg, delsarte, neum}
and more recently \cite{nonantipodal, rho=2, nested, binctrarb,kronprod,lifting}). Additionally, certain distance regular 
graphs can be described as coset graphs of completely regular codes \cite[p.353]{distreg}, and so 
such codes are also of interest to graph theorists. \emph{Completely transitive codes}, which are a subfamily of
completely regular codes with a high degree of algebraic symmetry, have also been studied (see \cite{BRZ,giudici,sole} for example).

For linear codes, the \emph{degree $s$} of a code is the number of values that occur as weights of
non-zero codewords. The \emph{dual degree $s^*$} of a linear code $C$ is the degree of 
its dual code $C^\perp$, where $C^\perp$ consists of all $\beta\in V(\Gamma)$ such that the dot product
$\alpha\cdot\beta := \sum_i \alpha_i\beta_i$ is equal to zero, for all codewords $\alpha\in C$.
The covering radius $\rho$ of a code $C$ is at most $s^*$, and $\rho=s^*$ if and only if $C$ is 
\emph{uniformly packed (in the wide sense)} \cite{remupc}. 
Completely regular codes are necessarily uniformly packed \cite{distreg}, but only a few constructions
of uniformly packed codes that are not completely regular are known \cite{kronprod}.

\subsection{The ternary Golay code} 
For a hypergraph or design $\De =(\Omega, \B)$, its \emph{incidence matrix} 
is the matrix whose columns are indexed by the points of $\Omega$, whose rows are indexed by 
the lines of $\B$, and such that the $(a,\ell)$-entry is 1 if the point $a$ lies in the line $\ell$,
and is zero otherwise. The row vectors are therefore binary $n$-tuples, where $n=|\Omega|$, and 
we may interpret their entries as elements of any field. For a field $F$ of order $q$, 
the code $C_{F}(\De)$  is defined as the linear span over $F$ 
of the rows of the incidence matrix of $\De$. It is contained in the Hamming graph $H(n,q)$.

In \cite{CEM}, the authors considered the code $C_{\mF_3}(\PG(2,3))$.  
They also constructed certain subcodes of this code, proving that one was 
the \emph{ternary Golay code}. Note that the ternary Golay code which Conway et al.~refer 
to is usually called the extended ternary Golay code in the coding theory literature. 
It is a $[12,6,6]_3$ code, which when punctured gives the perfect $[11,6,5]_3$ Golay code. 
We now describe their construction. 

Let $\mathcal{C} = C_{\mF_3}(\PG(2,3))$, and let $p\in \mathcal{P}$, the point set of $\PG(2,3)$. Conway et al.~define
$$
\mathcal{C}_p=\{\alpha\in \mathcal{C}\,|\,\alpha_p=-\sum_{i\in\mathcal{P}}\alpha_i\},
$$
and prove that the restriction of $\mathcal{C}_p$ to the coordinates 
$\mathcal{P}\backslash\{p\}$ is isomorphic
to the $[12,6,6]_3$ ternary Golary code \cite[Prop. 3.2]{CEM}, which has automorphism group $M_{12}$. 
It is this fact that is used by Conway et al.~ to prove that
$\pi_\infty(\PG(2,3))\cong M_{12}$, \cite[Thm. 3.5]{CEM}.

We now show that the full code $\mathcal{C}$ also has interesting properties which, to our knowledge,
have not been observed previously.

\begin{Thm}\label{thm:code1}
$\mathcal{C}$ is uniformly packed (in the wide sense), but not completely regular.
\end{Thm}

In order to prove this, we consider the following subcode of $\mathcal{C}$, which Conway 
et al.~ use to determine certain properties of $\mathcal{C}$:
$$
\mathcal{C}'=\{\alpha\in\mathcal{C}\,|\,\sum\alpha_i=0\}.
$$ 
For a line $\ell$ of $\PG(2,3)$, let $h_\ell$ denote the weight $4$ vector in $H(13,3)$ with 
$i$-entry equal to $1$ if $i\in\ell$, and zero otherwise.

\begin{Lem}\cite[Prop. 3.1]{CEM} \label{lem:conway1} Let $\alpha\in\mathcal{C}$. Then
\begin{itemize}
\item[i)] $\wt(\alpha)\equiv 0$ or $1\pmod 3$;
\item[ii)] $\alpha\in\mathcal{C}'$ if and only if $\wt(\alpha)\equiv 0\pmod 3$;
\item[iii)] $\mathcal{C}^\perp=\mathcal{C}'$;
\item[iv)] $\mathcal{C}$ and $\mathcal{C}'$ have minimum distance $4$, $6$, respectively, and the 
weight $4$ codewords in $\mathcal{C}$ are precisely the vectors $\pm h_\ell$, for lines $\ell$ of $\PG(2,3)$;
\item[v)] Let $\ell$ be a line in $\PG(2,3)$.  Then 
\begin{equation}\label{eq:codes1}
\sum_{i\in\mathcal{P}}\alpha_i=\sum_{i\in\ell}\alpha_i.
\end{equation}
\end{itemize}
\end{Lem}

We also need the following concepts. A vertex $\beta$ is said to \emph{cover} a vertex $\alpha$ if
$\beta_i=\alpha_i$ for all $i\in\supp(\alpha)$. So for example, $(2,1,1,0)$ covers $(2,1,0,0)$ in 
$\mF_3^4$. A set $\mathcal{S}$ of vertices of weight $k$ in $\mF_q^m$ is a 
\emph{$q-ary$  $t-(m,k,\lambda)$ design} if every vertex of weight $t$ in $\mF_q^m$ is covered by 
exactly $\lambda$ elements of $\mathcal{S}$. It is known that for a linear completely regular code $C$ in 
$\mF_q^m$ with minimum distance $\delta$, the set of codewords of weight $k$ forms a $q$-ary 
$\lfloor\frac{\delta}{2}\rfloor-(m,k,\lambda)$ design for some $\lambda$ \cite[Theorem 2.4.7]{vantilborg}.

\begin{proof}[Proof of Theorem \ref{thm:code1}]
First we show that $\mathcal{C}$ is uniformly packed.
By Lemma \ref{lem:conway1}, and since the codes have length $13$, the possible weights of non-zero 
codewords of $\mathcal{C}^\perp=\mathcal{C}'$ are $6,9$ and $12$, and hence $s^*\leq 3$. 
Also, since $\mathcal{C}$ has minimum distance $4$, the covering radius $\rho$ is at least $2$, so 
$2\leq\rho\leq s^*\leq 3$. Let $\ell_i$ for $i=1,2$ denote two of the four lines that contain 
the point $1$ in $\PG(2,3)$, and let $x, y$ be points
on $\ell_1,\ell_2$ respectively, that are distinct from $1$. Let $\nu$ be any vertex of 
weight $3$ with $\supp(\nu)=\{1, x, y\}$. 
Since $\PG(2,3)$ is a projective plane, it follows that $|\supp(\nu)\cap\ell |\leq 2$ for 
all lines $\ell$ of $\PG(2,3)$. Thus, Lemma \ref{lem:conway1}(iv) implies that 
$d(\nu,\alpha)\geq 3$ for all codewords $\alpha$ of weight $4$. Since all other 
codewords in $\mathcal{C}$ have weight at least $6$, it followse that $d(\nu,\mathcal{C})\geq 3$, 
and hence $\rho=s^*=3$. Thus $\mathcal{C}$ is uniformly packed.

Suppose that $\mathcal{C}$ is completely regular. Then the set of codewords of weight $4$ forms 
a $3$-ary $2-(13,4,\lambda)$ design for
some $\lambda>0$. However, by Lemma \ref{lem:conway1}(iv), every vertex of weight $2$ with 
constant non-zero entries is covered by exactly one codeword of weight $4$, 
whereas a vertex of weight $2$ with distinct non-zero entries is not covered by any codeword of weight $4$.
This contradiction proves that  $\mathcal{C}$ is not completely regular.
\end{proof}

\subsection{Conway groupoids and completely regular codes.} As we have seen, the Conway groupoid 
$M_{13}$ is interesting in several different ways. Its hole stabilizer $\pi_\infty(\mathcal{D})$ 
is multiply transitive, and hence primitive; the perfect Golay code over $\mF_3$ can be constructed 
from $C_{\mF_3}(\PG(2,3))$; and, moreover, the code $C_{\mF_3}(\PG(2,3))$ has some interesting 
and rare properties. Thus it is natural to ask if one can construct other interesting codes
from supersimple designs $\mathcal{D}$ for which $\pi_\infty(\mathcal{D})$ is acting primitively. 
This question was addressed in \cite{GGS} for the designs defined in Examples \ref{ex: symplectic} 
and \ref{ex: orthogonal}.

\begin{Thm}\label{thm:codesct} The following hold:
\begin{itemize}
\item[(a)] For $\ep \in \mathbb{F}_2$, $C_{\mF_2}(\mathcal{D^\ep})$ is a completely 
transitive code with covering radius $3$ and minimum distance $4$.  
\item[(b)] $C_{\mF_2}(\mathcal{D}^a)$ is a completely transitive code with covering 
radius $4$ and minimum distance $4$. 
\end{itemize}
\end{Thm}

The following question arises naturally as a consequence of Theorems \ref{thm:code1} and \ref{thm:codesct}.

\begin{Qu} Let $\mathcal{D}=(\Omega,\B)$ be a supersimple $2-(n,4,\lambda)$ design 
such that $\pi_\infty(\mathcal{D})$ is a primitive subgroup $\Sym(\Omega\setminus\{\infty\})$ 
which does not contain $\Alt(\Omega\setminus\{\infty\})$. 
Does there exist a prime $r$ such that $C_{\mF_r}(\mathcal{D})$ is completely regular, 
or at the very least uniformly packed (in the wide sense)?
\end{Qu}

We remarked earlier that for each point  $p \in \PG(2,3)$, Conway et al define a code 
$\C_p$ with the property that $\Aut(\C_p) \cong M_{12}$. They use this code to show that 
$\pi_\infty(\PG(2,3)) \cong M_{12}$: namely, in \cite[Proposition 3.3]{CEM} they show 
that the elementary move $[p,q]$ sends $\C_p$ to $\C_q$, and from this they deduce in
\cite[Proposition 3.4]{CEM} that $\pi_\infty(\PG(2,3)) \leq M_{12}$. Equality then follows 
by an explicit computation. Arguing in this spirit, with $\C$ being one of the codes 
$\C_{\mathbb{F}_2}(\De)$ of Theorem \ref{thm:codesct}, it is relatively straightforward 
to show that for and point $p$ of $\De$, the code $\C_p$ obtained by puncturing $\C$ at 
$p$ has automorphism group isomorphic to the stabilizer $\Stab_{\Aut(\De)}(p)$ of $p$. 
Moreover, the elementary move $[p,q]$ sends $\C_p$ to $\C_q$, and we deduce that 
$\pi_\infty(\De) \leq \Stab_{\Aut(\De)}(p)$. This fact can be used to give an 
alternative proof (to that given in \cite{GGS}) of the isomorphism type of $\L_\infty(\De)$ 
for the designs $\De$ defined in Examples \ref{ex: symplectic} and \ref{ex: orthogonal}.

\section{Classification results}

The programme to classify Conway groupoids has, thus far, been restricted to the situation 
where $\De$ is a supersimple $2-(n,4,\lambda)$ design, a family which includes $\PG(2,3)$. 
In this section we describe the progress that has been made in this setting.

\subsection{Relation between \texorpdfstring{$n$}{n} and \texorpdfstring{$\lambda$}{lambda}}

In this subsection we connect the relative values of the parameters $\lambda$ and $n$ with the 
behaviour of the hole stabilizer $\pi_\infty(\De)$ in its action on $\Omega\backslash\{\infty\}$.
By examining the examples given in \S\ref{s: examples} one may be lead to observe the following: 
if we fix $\lambda$ and allow $n$ to increase, the way $\pi_\infty(\De)$ acts on 
$\Omega\setminus\{\infty\}$  seems to move through the following states:
\[
 \textrm{trivial} \longrightarrow \textrm{intransitive} \longrightarrow \begin{array}{l}
                                                                         \textrm{transitive} \\ \textrm{imprimitive}
                                                                        \end{array} \longrightarrow \textrm{primitive}
\longrightarrow  \begin{array}{l}
                        \Alt(\Omega\setminus\{\infty\}) \ \textrm{ or }  \\ \Sym(\Omega\setminus\{\infty\}).
                                                                        \end{array} 
\]
\noindent
This observation was proved and quantified in \cite{GGNS,GGS}. We note first that two points 
in a $2-(n,4,\lambda)$ design lie together on $\lambda$ lines, and the set theoretic union of 
these lines in a supersimple  design has size $2\lambda + 2$. Thus for supersimple designs we 
must have $n\geq 2\lambda+2$. On the other hand, Lemma~\ref{lem:nlambda} gives an upper bound 
in terms of $\lambda$ beyond which, for all designs $\De$, $\pi_\infty(\De)$ is 
$\Alt(\Omega\setminus\{\infty\})$ or $\Sym(\Omega\setminus\{\infty\})$. This bound is refined 
in the result quoted below, and we make some comments about the proof in Remark~\ref{rem:n and lambda}.

\begin{Thm}\label{t: n and lambda} (\cite[Theorem B]{GGNS} and \cite[Theorem E]{GGS})
 Suppose that $\De$ is a supersimple $2-(n,4,\lambda)$ design, and $\infty$ is a point of $\De$.
 \begin{enumerate}
  \item If $n>2\lambda+2$, then $\pi_\infty(\De)$ is non-trivial;
  \item if $n>4\lambda+1$, then $\pi_\infty(\De)$ is transitive;
  \item if $n> 9\lambda+1$, then $\pi_\infty(\De)$ is primitive;
\item  if $n>9\lambda^2-12\lambda+5$, then either 
$\pi_\infty(\De)\subseteq \Alt(\Omega\setminus\{\infty\})$, or else $\De=\PG(2,3)$, $\pi_\infty(\De) =  M_{12}$ (and $\lambda=1$).
\end{enumerate}
\end{Thm}

\begin{Rem}\label{rem:n and lambda}{\rm 
(a) The proofs of the first three parts of Theorem~\ref{t: n and lambda} are independent of the 
Classification of the Finite Simple Groups (CFSG) (as is that of Lemma~\ref{lem:nlambda}), but 
this is not true for part (4).

(b) Part (1) can be strengthened: in \cite[Theorem B]{GGNS} it was shown that  $\pi_\infty(\De)$ 
is trivial if and only if $\De$ is a Boolean quadruple system, that is, one of the designs from 
Example~\ref{ex: boolean}.

(c) While the proofs of parts (2) and (3) are relatively straightforward counting arguments, 
the proof of Lemma~\ref{lem:nlambda} lies somewhat deeper. It relies on a lower bound proved by 
Babai~\cite{Babai} for the minimum number of points moved by a non-identity element of a 
primitive permutation group that does not contain the full alternating group. Babai's bound 
is combined with the observation that a move sequence $[\infty, a, b, \infty]$ will have 
support of size at most $6\lambda+2$; now one must check that there exists such an element 
that is non-trivial, and the result follows.

(d)  Part (4) (which is an improvement on Lemma~\ref{lem:nlambda}) is obtained via the same 
method except that only those move sequences $[\infty, a, b, \infty]$ for which $a,b$ and 
$\infty$ are collinear are considered, and the result of Babai is replaced by a stronger 
result due to Liebeck and Saxl \cite{LS}; it is here that the dependence on CFSG enters. 
 
(e) It is natural to ask whether the bounds in  Lemma~\ref{lem:nlambda} and 
Theorem~\ref{t: n and lambda} are best possible. Certainly part (1) cannot be 
improved, but for the others it is less clear.
 }
\end{Rem}

\begin{Qu}\label{q: linear}
Can the quadratic function in Lemma~\ref{lem:nlambda} be replaced by a linear function?
\end{Qu}

An immediate corollary of Theorem~\ref{t: n and lambda} is the following.

\begin{Cor}\label{c: finite iso}
For a positive integer  $\lambda$, there are only finitely many supersimple $2-(n,4,\lambda)$ 
designs $(\Omega, \B)$ for which $\L(\De)$ does not contain $\Alt(\Omega)$. 
\end{Cor}

This corollary suggests that a full classification for a given $\lambda$ may be possible. 
This has been achieved for $\lambda \leq 2$ in \cite[Theorem C]{GGNS}, but all other cases are open.

\begin{Thm}\label{t: small lambda} \cite[Theorem C]{GGNS}
Let $\De=(\Omega,\B)$ be a supersimple $2$-$(n,4,\lambda)$ design for which $\L(\De)$ 
does not contain $\Alt(\Omega)$, and such that $\lambda \leq 2$. Let $\infty\in\Omega$. Then either
\begin{enumerate}
\item $\lambda=1$, $\De=\PG(2,3)$ and $\pi_\infty(\De)=M_{12}$; or
\item  $\lambda=2$, $\De$ is the unique supersimple $2-(10,4,2)$ design and $\pi_\infty(\De)=\Sym(3)\wr\Sym(2)$.
\end{enumerate}
\end{Thm}

The design in Theorem~\ref{t: small lambda}~(2) is connected to the family of designs in Example~\ref{ex: orthogonal} (recall that $\Sym(3)\wr\Sym(2)\cong{\rm O}_4^+(2)$ and see the remark after Example~\ref{ex: orthogonal}).
Theorem~\ref{t: small lambda} pre-dates Theorem~\ref{t: n and lambda}, but its proof is of a 
similar flavour. In this case, the result of Babai mentioned in Remark~\ref{rem:n and lambda} 
is replaced by classical work of Manning classifying primitive permutation groups that contain 
non-identity elements moving less than $9$ points, \cite{manning, manning1}.

\begin{Qu}\label{q: specific lambda}
 Can those $2-(n,4,\lambda)$ designs $\De$ be classified for which $\L(\De)$ does not contain 
$\Alt(\Omega)$ and $\lambda$ is, say, $3,4$ or $5$?
\end{Qu}

Some remarks concerning a classification for $\lambda=3$ can be found in \cite[\S7.3]{GGS}.

\subsection{Extra structure}
In this section we consider two instances where we have been able to give a complete 
classification of Conway groupoids $\L_\infty(\De)$ subject to some set of conditions 
on the elementary moves, for supersimple designs $\De$. 

\subsubsection{Collinear triples yielding trivial move sequences} Here we consider 
\cite[Theorem D]{GGS} which was a critical ingredient in the proof of 
Theorem \ref{t: n and lambda} (4) above, and which generalizes the 
classification of Conway groupoids associated with $2-(n,4,1)$ designs given in Theorem \ref{t: small lambda}.

\begin{Thm}\label{t: d}
Suppose that $\De$ is a supersimple $2-(n,4,\lambda)$ design, and that $[\infty, a, b, \infty]=1$ whenever $\infty$ is collinear with $\{a,b\}$. Then one of the following is true:
\begin{enumerate}
 \item $\De$ is a Boolean quadruple system and $\pi_\infty(\De)$ is trivial;
 \item $\De=\PG(2,3)$ (the projective plane of order $3$) and $\pi_\infty(\De) \cong M_{12}$; or 
 \item $\pi_\infty(\De)\supseteq \Alt(\Omega\setminus\{\infty\})$.
\end{enumerate}
\end{Thm}

Recall that a Boolean quadruple system was defined in Example~\ref{ex: boolean}. 
The proof of Theorem~\ref{t: d} given in \cite{GGS} involves an interesting intermediate result, \cite[Proposition 6.4]{GGS}. This result asserts that any design $\De$ which satisfies the hypotheses 
of Theorem \ref{t: d} can be constructed in a rather curious way: one starts with 
a $2-(n,2^{\alpha+1},1)$ design $\De_0$ (for some $\alpha \in \Z^+$) and one 
``replaces'' each line in $\De_0$ with a Boolean quadruple system of order 
$2^{\alpha+1}$. One thereby obtains a $2-(n,4,2^{\alpha}-1)$ design satisfying 
the given hypothesis, and all such designs arise in this way. 


\subsubsection{Regular two-graphs}\label{s: two graphs} In this section we study three properties which turn out 
to be connected in the context of Conway groupoids.

Firstly, a $2-(n,3,\mu)$ design $(\Omega,\C)$ is a \textit{regular two-graph} if, for 
any 4-subset $X$ of $\Omega$, either $0, 2$ or $4$ of the $3$-subsets of $X$ lie in 
$\C$. We are interested in those $2-(n,4,\lambda)$ designs $\De$ for which the pair 
$(\Omega, \C)$ is a regular two-graph, where $\C$ is the set of triples of collinear points. 

Secondly, we consider designs $\De=(\Omega, \B)$ that satisfy the following property: 
\begin{equation}\label{e:symdiff}
\textrm{if }B_1,B_2 \in \B \textrm{ such that } |B_1 \cap B_2|=2, \textrm{ then } 
B_1 \triangle B_2 \in \B \tag{$\triangle$}
\end{equation}
where $B_1\triangle B_2$ denotes the {\it symmetric difference} of  $B_1$ and $B_2$.

Finally, we are interested in those designs for which $\L_\infty(\De)$ is a group. 
The following result which is (part of) \cite[Theorems A and 4.2]{GGPS} connects 
these three properties. It is proved combinatorially.

\begin{Thm}\label{t:l8gp}
Let $\De=(\Omega,\B)$ be a supersimple $2-(n,4,\lambda)$ design with $n > 2\lambda+2$. 
Let $\C$ denote the set of collinear triples of points in $\Omega$, and let $\infty 
\in \Omega$. Then the following hold.
\begin{itemize}
\item[(a)] If $\L_\infty(\De)$ is a group then $\L_\infty(\De)$ is primitive on $\Omega$.
\item[(b)] If $(\Omega,\C)$ is a regular two-graph then $\pi_\infty(\De)$ is transitive 
on $\Omega\setminus\{\infty\}$.
\item[(c)] If $(\Omega,\C)$ is a regular two-graph and $\L_\infty(\De)$ is a 
group then $\pi_\infty(\De)$ is primitive on $\Omega\setminus\{\infty\}$.
\item[(d)] If $(\Omega, \C)$ is a regular two-graph and $\De$ satisfies \eqref{e:symdiff}, then $\L_\infty(\De)$ is a group.
\end{itemize}
\end{Thm}

It turns out that part (d) can be strengthened: in \cite[Theorem 4.2]{GGPS}) we show 
that the group $\L_\infty(\De)$ is in fact a subgroup of automorphisms of $\De$, and is 
a 3-transposition group with respect to its set $\E$ of elementary moves. This 
observation was combined in \cite{GGPS} with Fischer's classification of finite 
3-transposition groups (\cite{fischer}) to classify Conway groupoids arising from 
designs $\De$ that satisfy the hypotheses of part (d).

The conditions of part (d) were also used in \cite{GGPS} in another way: it turns 
out that, for any point $\infty \in \Omega$, the assumptions of part (d), together 
with the condition $n>2\lambda+2$, imply that $(\Omega\backslash\{\infty\}, \C_{\infty})$ 
is a polar space in the sense of Buekenhout and Shult, where $\C_\infty$ is the set of 
all triples of points in $\Omega \backslash \{\infty\}$ which occur in a line with $\infty$. 
In fact, the polar space $(\Omega\backslash\{\infty\}, \C_\infty)$ has the extra 
property that all lines in the space contain exactly three points. Such polar spaces 
were characterized in a special case by Shult \cite{Sh} and then later, in full generality, 
by Seidel \cite{Seidel}. Seidel's result was used to derive the following classification result. 
This result provides an alternative proof for the classification of the associated 
Conway groupoids, avoiding the use of 3-transposition groups.

\begin{Thm}\label{t:ggpsb}\cite[Theorem C]{GGPS}
Let $\De=(\Omega,\B)$ be a supersimple $2-(n,4,\lambda)$ design that satisfies  
$\eqref{e:symdiff}$ and for which $(\Omega, \C)$ is a regular two-graph where $\C$ 
is the set of collinear triples of points in $\Omega$. Then one of the following holds:
\begin{itemize}
\item[(a)] $\De$ is a Boolean quadruple system, as in Example~\ref{ex: boolean};
\item[(b)] $\De$ is a Symplectic quadruple system, as in Example~\ref{ex: symplectic};
\item[(c)] $\De$ is a Quadratic quadruple system, as in Example~\ref{ex: orthogonal}.
\end{itemize}
\end{Thm} 

Note that the structures of the corresponding hole stabilizers and Conway groupoids are 
listed in the relevant examples. One naturally wonders if this theorem can be strengthened:

\begin{Qu}
Can Theorem \ref{t:ggpsb} be extended to cover the situation where ($\triangle$) does not hold?
Are there any additional examples?
\end{Qu}




We conclude by noting that the statements of Theorems~\ref{t: n and lambda} and \ref{t: d} both require particular clauses to deal with $M_{13}$: it seems that, in the world of Conway groupoids, $M_{13}$ is rather special. The following question connects this notion to the study of Conway groupoids with extra structure.

\begin{Qu}
Is $M_{13}$ the only Conway groupoid which is not itself a subgroup of $\Sym(\Omega)$, 
and for which the associated hole stabilizer $\pi_\infty(\Omega)$ is a primitive subgroup 
of $\Sym(\Omega\setminus\{\infty\})$?
\end{Qu}

Note that the Conway groupoids arising in Examples~\ref{ex: symplectic} and 
\ref{ex: orthogonal} have primitive hole stabilizers, but \emph{are} subgroups of $\Sym(\Omega)$.

\section{Generation games}


We have seen that a supersimple $2-(n,4,\lambda)$ design provides a convenient structure by 
which to associate with each pair $\{a,b\}$ of points a permutation $[a,b]$ sending $a$ to $b$.
We conclude this survey by considering a few other combinatorial structures which might be 
exploited to find interesting new Conway groupoids. 

\subsection{Working with triples}\label{triples}
For a $2-(n,3,\mu)$ design $(\Omega,\C)$, a map $[\cdot, \cdot]: \Omega \times \Omega 
\longrightarrow \Sym(\Omega)$ is said to be a \textit{pliable function associated with $(\Omega,\C)$} 
if the following hold:
\begin{itemize}
\item[(a)] for each $a,b \in \Omega$, $[a,b]$ sends $a$ to $b$ and $[a,b]^{-1}=[b,a]$;
\item[(b)] for $a\ne b$, $\supp([a,b]) = \{a,b\} \cup \{c \mid c \mbox{ is collinear with $a,b$} \}$.
\end{itemize}
Here $\supp(g)$ (for $g\in\Sym(\Omega)$) means the set of points of $\Omega$ moved by $g$,
and a point $c$ is \emph{collinear} with $\{a,b\}$ if $\{a,b,c\}\in\C$. 
We usually assume also that $[a,a]=1$ for all $a\in\Omega$.

For such a function, and for each $a_0,a_1,\dots, a_k\in \Omega$, define:
\[
[a_0,a_1,a_2,\ldots,a_k]:=[a_0,a_1][a_1,a_2]\cdots[a_{k-1},a_k],
\]
to be a {\it move sequence} and for each $\infty\in \Omega$, define:
\begin{equation}\label{eq:cg}
 \L_\infty([\cdot,\cdot]):= \{ [\infty,a_1,a_2,\ldots,a_k] \mid k \in \Z, a_1,\dots, a_k \in \Omega \}\subseteq \Sym(\Omega); \mbox{ and}
\end{equation} 
\begin{equation}\label{eq:hs}
\pi_\infty([\cdot,\cdot]):=  \{ [\infty,a_1,a_2,\ldots,a_{k-1}, \infty] \mid k \in \Z, a_1,\dots, a_{k-1} \in \Omega\} \subseteq \Sym(\Omega \backslash \{\infty\})
\end{equation} 
to be the \textit{Conway groupoid} and \textit{hole-stabilizer}, respectively, associated with $\infty$. 
We have the following examples.
\begin{enumerate}
 \item[(a)]  A supersimple $2-(n,4,\lambda)$ design $\De=(\Omega,\B)$ determines a 
$2-(n,3,2\lambda)$ design $(\Omega, \C)$, where $\C$ is the set of collinear triples of $\De$. 
The elementary moves associated with $\De$ determine a pliable function 
$[\cdot,\cdot]: \Omega \times \Omega \longrightarrow \Sym(\Omega)$ associated with $(\Omega, \C)$. 
Moreover $\L_\infty(\De)=\L_\infty([\cdot,\cdot])$ and $\pi_\infty(\De)=\pi_\infty([\cdot,\cdot])$, 
for each $\infty \in \Omega$. 

\item[(b)] \textit{Any} finite group $G$ determines a pliable function $[\cdot,\cdot]: G \times G 
\longrightarrow \Sym(G)$ associated with $(G, \C)$, where  $\C$ is the set of all $3$-subsets of $G$, 
by taking $[a,b]$ to be right multiplication by $a^{-1}b$. Thus $[a,b]$ is the unique element
of the right regular action of $G$ on $G$ which maps $a$ to $b$. Here $(\Omega, \C)$
is a $2-(n,3,n-2)$ design, where $n=|G|$, and $\L_\infty([\cdot,\cdot])\cong G$, $\pi_\infty([\cdot,\cdot])=1$. Observe that by (a) the Boolean $2-(2^m,4,2^{m-1}-1)$ designs of Example \ref{ex: boolean} determine pliable functions of this type where $G \cong (C_2)^m$.

\item[(c)] For an example which is not of either of these types consider the unique 
$2-(6,3,2)$ design $(\Omega, \C)$ whose lines are given by:
$$
\begin{array}{ccccc}
012 & 023 & 034 & 045 & 051 \\ 
124 & 235 & 341 & 452 & 513 \\
\end{array}
$$
Thus $(\Omega, \C)$ is the (extended) Paley two-graph with automorphism group ${\rm PSL}(2,5)$. 
For $a,b \in \Omega$, let $[a,b]:=\Id_\Omega$ if $a=b$ and otherwise set $[a,b]:=(a,b)(c,d)$ 
where $\{a,b,c\}$ and $\{a,b,d\}$ are the two lines containing $\{a,b\}$ in $\C$. Then 
$[\cdot,\cdot]$ becomes a pliable function associated with $(\Omega, \C)$, and it is easy 
to show that $\L_\infty([\cdot,\cdot])=\Aut(\C)= {\rm PSL}(2,5)$.

\item[(d)] More exotic examples arise also. For example the Higman--Sims sporadic simple group $HS$ has a $2$-transitive action on a set $\Omega$ of degree $176$, and $\Omega$ forms the point set of a $2-(176,3,162)$ design admitting $HS$ as a group of automorphisms. Moreover the setwise stabiliser in $HS$ of an unordered pair $\{a,b\}$  of distinct points has a unique central involution $z_{a,b}$.  These involutions form a  conjugacy class of $HS$ of size $15400 = \binom{176}{2}$ and, furthermore, the map $[\cdot,\cdot]: \Omega \times \Omega$ given by $[a,b]=z_{a,b}$ is a pliable function,
yielding $\L_\infty([\cdot,\cdot])=HS$. 
We are grateful to Ben Fairbairn for informing us of this example.

\end{enumerate}

Notice that $\L_\infty([\cdot,\cdot])$ formed a subgroup of $\Sym(\Omega)$ in several 
of the above examples. Under this assumption, we can prove the following:

\begin{Thm}\label{t:linf}
Let $[\cdot,\cdot]$ be a pliable function associated with a \textit{$2-(n,3,\mu)$ design} 
$(\Omega, \C)$, where $\mu > 4$, and suppose that $\L_\infty([\cdot,\cdot])$ is a group. 
If $n > \frac{3}{2}\mu$, then $\L_\infty([\cdot,\cdot])$ is primitive on $\Omega$.
\end{Thm}

\begin{proof}
Suppose that $n> \frac{3}{2}\mu$ and $\mu\geq4$. Then $n>\mu +\frac{1}{2}\mu \geq \mu+2$,
and we note that for distinct $a,b$, $|\supp([a,b])|=\mu+2$, by the definition of a pliable function.
Suppose  that $G$ acts imprimitively on $\Omega$ with $m$ blocks of size $k$, where 
$n=mk$ and $m>1, k>1$. First we observe that $m\geq 3$. This holds because, if $m=2$, then 
for points $a, b$ in different blocks of imprimitivity, the elementary move $[a,b]$ must 
interchange the two blocks, and hence $[a,b]$ must move every point, contradicting the fact
that $\supp([a,b])=\mu+2 < n$. Now let $a,b$ be distinct points in the same block of 
imprimitivity $\Delta$, and let $y$ be any point fixed by $[a,b]$ (such a point exists 
since $n> \mu+2$). By part (b) of the definition of a pliable function, it follows that 
$g:=[a,y]$ fixes $b$, so $g$ must fix $\Delta$ setwise, and hence $y=a^g \in \Delta$. 
This shows that every point fixed by $[a,b]$ lies in $\Delta$, or equivalently that  
$\Omega \setminus \Delta \subseteq \supp([a,b]).$ 
Thus
$$
\mu+2 = |\supp([a,b])| \geq (m-1)k+2 \quad\mbox{and hence}\quad \mu \geq (m-1)\cdot \frac{n}{m}.
$$ 
Rearranging this yields  
\begin{equation}\label{eq:mbound}
n\leq\frac{m}{(m-1)}\cdot \mu \leq\frac{3}{2} \cdot \mu
\end{equation} 
and this contradiction completes the proof.
 \end{proof}

The bound given in Theorem \ref{t:linf} is achieved by at least one design. 
To see this we construct a pliable function for a $2-(9,3,6)$ design $(\Omega, \C)$ 
with the property that $\L_\infty([\cdot,\cdot])$ is transitive but imprimitive. 
Let $\Omega:=(\mathbb{F}_3)^2$ and let $\C$ be the complement of an affine plane of order 3,
that is, 
$$
\C:=\{\{a,b,c\} \mid a, b, c \in \Omega,  a+b+c \neq 0 \}.
$$ 
For not necessarily distinct points $a,b \in \Omega$, set 
$$
[a,b]:=\prod_{w+a+b \neq 0} (w,a+b-w).
$$ 
Then $[a,b]$ is an involution with support of size eight and a unique fixed point $w=-a-b$. 
For each $\infty \in \Omega,$ $\L_\infty([\cdot,\cdot]) \cong (C_3 \times C_3): C_2$ with 
the nine non-trivial involutions given by $\{[a,b] \mid a,b \in \Omega\}$ (notice that 
$[a,b]=[c,d]$ whenever $a+b=c+d$). Furthermore, it is easy to see that  
$\L_\infty([\cdot,\cdot])$ preserves a system of imprimitivity with three blocks of size 3.

In fact this example is just one of an infinite family of $2-(3^k,3,3^k-3)$ designs 
constructed from complements of affine spaces with the property that they admit pliable 
functions with Conway groupoid an imprimitive group \cite{GGPS2}. Just as the Boolean quadruple systems 
in Example \ref{ex: boolean} provided the ``smallest'' examples of designs satisfying the 
hypotheses of Theorems \ref{t: d} and \ref{t:ggpsb}, one might hope that these  
$2-(3^k,3,3^k-3)$ designs could play a similar role in this more general context. For this reason, we ask the following:

\begin{Qu}\label{c:lowerbound}
Let $[\cdot,\cdot]$ be a pliable function associated with a \textit{$2-(n,3,\mu)$ design} 
$(\Omega, \C)$ and suppose that $\L_\infty([\cdot,\cdot])$ is a group. If $n > \mu+3$, 
is $\L_\infty([\cdot,\cdot])$ primitive?
\end{Qu}

\subsection{Using \texorpdfstring{$4$-hypergraphs}{4-hypergraphs}}
As discussed in Section~\ref{s: general}, most of the interesting Conway groupoids known arise from 
$2-(n,4,\lambda)$ designs. We gave one alternative approach in Subsection~\ref{triples} based on triple 
systems. Here we discuss briefly a few other possibilities involving $4$-hypergraphs which are not $2$-designs.
The following infinite family of examples was presented in \cite[Example 4.1]{GGNS}.

\begin{Ex}\label{ex5.2}
Let $n\geq3$, let  $\Omega$ be a set of size $2n$ consisting of the points 
$\{x_i, y_i \ \mid\ 1\leq i\leq n\}$, and let $\B$ be the set 
\[
 \B:=\{\ \{x_i,y_i,x_j,y_j\}\ \mid\ 1\leq i < j\leq n\}. 
\]
Then $\De :=(\Omega,\B)$ is a connected, pliable $4$-hypergraph, and, for any $\infty\in\Omega$, 
the Conway groupoid $\L_\infty(\De)$ and hole stabilizer $\pi_\infty(\De)$ are defined as 
in the first part of Section~\ref{s: general}. It was noted in \cite[Example 4.1]{GGNS} that $\pi_\infty(\De) \cong \Sym(2)\wr \Sym(n-1)$
if $n$ is odd, and that $\pi_\infty(\De)$ is an index $2$ subgroup of 
$\Sym(2)\wr \Sym(n-1)$ if $n$ is even. We give a short proof of this and also show that 
$G:=\L_\infty(\De)$ is a group, equal to $\Sym(2) \wr \Sym(n)$ if $n$ is odd, and its
the index $2$ subgroup $(\Sym(2) \wr \Sym(n))\cap\Alt(n)$ if $n$ is even.

The elementary moves are: for distinct $i,j\in\{1,\dots,n\}$,
\begin{equation}\label{eltmv}
[x_i,x_j] = [y_i,y_j] = (x_i,x_j)(y_i,y_j) \mbox{ and } [x_i,y_j]=[y_i,x_j] = (x_i,y_j)(x_j,y_i),  
\end{equation}
together with the fixed point free involution $[x_1,y_1] = \dots = [x_n,y_n] = \prod_{i=1}^n (x_i,y_i)$.
Since $\De$ is connected, we may assume that $\infty:=x_1$. It is readily checked that 
for each triple of elements $a,b,c \in \Omega$ we have 
$$
[a,b]^{[b,c]} = [a^{[b,c]},c].
$$ 
Hence an argument in \cite[Lemma 2.7]{GGPS} shows that $G$ is a group. 
Moreover \cite[Lemma 2.6]{GGPS} implies that $G_0:=\pi_\infty(\De)=\stab_G(\infty)$. 
(Note that, although both of the cited results in \cite{GGPS} are stated and proved for 
supersimple designs, in fact the argument carries through for connected, pliable 
$4$-hypergraphs.) 

Next we see that $G$ leaves invariant the system of imprimitivity $\Delta$ given by 
$$
\Delta:=\{\{x_i,y_i\} \mid 1 \leq i \leq n \}.
$$ 
Hence $G_0$ fixes the block $\{x_1,y_1\}$ and we have 
\[
G \leq \Sym(2) \wr \Sym(n) \mbox{ and } G_0 \leq \Sym(2) \wr \Sym(n-1). 
\]
Since $G_0=\stab_G(\infty)$, $G_0$ contains the 
elementary moves $[x_i,x_j], [x_i,y_j]$ given in \eqref{eltmv}, for each $i, j$ such that $2\leq i<j\leq n$, 
and moreover, $G_0$ contains the product of these two elements which is $(x_i,y_i)(x_j,y_j)$. 
Thus $G_0$ induces $\Sym(n-1)$ on $\{\{x_i,y_i\}\,\mid \, i=2,\dots, n\}$. Indeed $G_0\cap 
\Alt(2n-2)$ induces $\Sym(n-1)$.

Now let $K \cong \Sym(2)^{n}$ denote the base group of the wreath product. Then, for distinct $i, j$, 
$K$ contains $[x_i,x_j] [x_i,y_j] = (x_i,y_i)(x_j,y_j)$, and it follows 
that  $G \cap K$ contains all the even permutations 
in $K$. Together these points imply first that $G_0$ contains 
\[
\left(\Sym(2) \wr \Sym(n-1)\right)\cap \Alt(2n-2). 
\]
In particular, $G_0$ has index at most $2$ in $\Sym(2)\wr\Sym(n-1)$. If $n$ is even then 
every product of elementary moves is an even permutation, so that $G_0$ is as claimed. If 
$n$ is odd, then $G_0$ also contains 
$$
g:=[x_1,x_2,y_1,x_1] =  [x_1,x_2] [x_2,y_1][y_1,x_1] = (x_1,y_1)(x_2,y_2)[y_1,x_1]
$$ 
which is an odd permutation since $[y_1,x_1] = \prod_{i=1}^n (x_i,y_i)$ is an odd permutation. 
Thus in this case $G_0$ is the full wreath product $\Sym(2)\wr\Sym(n-1))$. 
Very similar arguments confirm the claims about $G$ (note that $|G|=|\Omega|\,|G_0|$).
\end{Ex}

The $4$-hypergraphs in this family are not $2$-designs since, for example, the pair $\{x_1, x_2\}$ lies in a unique 
line, while $\{x_1,y_1\}$ lies in $n-1$ lines. On the other hand, every pair of points is contained in 
at least one line. Hypergraphs with this property are said to be \emph{collinearly complete}, (see \cite{ABP}; 
their study goes back to work of D. G. Higman and J. E. McLaughlin in \cite{HM}). 

\begin{Qu}\label{ccdesigns}
Are there other interesting familes of Conway groupoids arising from collinearly complete $4$-hypergraphs which are not $2$-designs? 
\end{Qu}

The family of connected pliable $4$-hypergraphs extends beyond those which are collinearly complete. It includes, 
for example, generalised quadrangles with $4$ points on each line. There are only finitely many such geometries, and it is shown 
in \cite{GGPS2} that for each of them the Conway groupoid is the full alternating group.

\begin{Qu}
Are there interesting Conway groupoids arising from connected pliable $4$-hypergraphs which are not collinearly complete? 
\end{Qu}

\subsection{\texorpdfstring{$M_{24}$}{M24}}

In the previous two subsections we have started with different geometries, and sought to ``play'' analogues of Conway's original ``game'' in order to obtain groups and\,/\,or groupoids. 

What about if one works backwards, that is to say, one starts with a group and seeks to define a game on an appropriate geometry that generates it? As we described at the start of this paper, this was Conway's original approach: he came to define his game after observing certain structural coincidences between the groups $\PSL_3(3)$ and $M_{12}$. 

In fact this structural coincidence can precisely be described as a `3-local equivalence' (in the sense that $\PSL_3(3)$ and $M_{12}$ have isomorphic $3$-fusion systems) and one immediately wonders whether there are other (pairs of) $p$-locally equivalent groups whose structure can be exploited in some similar fashion to give a ``natural'' generation game. 

For example, might the group $M_{24}$ be amenable to such an analysis, perhaps via some analogue of the dualized game (described in \S\ref{s: variants}) played on an appropriate geometry? Might there exist a Conway groupoid $M_{25}$ -- or perhaps, as Conway himself mentioned after a lecture given by the third author -- might there be an $M_{26}$? In this direction, the 3-local equivalence between $M_{24}$ and $\PSL_3(3):2$ is particularly suggestive (note that the latter can be realised as a group of permutations on 26 letters). In any case, a generation game for $M_{24}$ would be immensely interesting and would naturally lead one to wonder about the other sporadic simple groups.

\begin{Qu}
 Can the group $M_{24}$ be generated in a natural way via a generation game on some finite geometry?
\end{Qu}


\begin{thebibliography}{10}

\bibitem{ABP}
Seyed~Hassan Alavi, John Bamberg, and Cheryl~E. Praeger, \emph{Triple
  factorisations of the general linear group and their associated geometries},
  Linear Algebra Appl. \textbf{469} (2015), 169--203.

\bibitem{Babai}
Laszlo {Babai}, \emph{{On the order of uniprimitive permutation groups}}, {Ann.
  Math. (2)} \textbf{113} (1981), 553--568.

\bibitem{remupc}
L.~A. Bassalygo and V.~A. Zinov{\cprime}ev, \emph{A remark on uniformly packed
  codes}, Problemy Pereda\v ci Informacii \textbf{13} (1977), no.~3, 22--25.
  \MR{0497283}

\bibitem{BQ}
Arrigo {Bonisoli} and Pasquale {Quattrocchi}, \emph{{Each invertible sharply
  $d$-transitive finite permutation set with $d\geq 4$ is a group}}, {J.
  Algebr. Comb.} \textbf{12} (2000), no.~3, 241--250.

\bibitem{nonantipodal}
J.~Borges, J.~Rif{\`a}, and V.~A. Zinoviev, \emph{On non-antipodal binary
  completely regular codes}, Discrete Math. \textbf{308} (2008), no.~16,
  3508--3525.

\bibitem{rho=2}
J.~Borges, J.~Rif{\`a}, and V.~A. Zinoviev, \emph{On {$q$}-ary linear
  completely regular codes with {$\rho=2$} and antipodal dual}, Adv. Math.
  Commun. \textbf{4} (2010), no.~4, 567--578.

\bibitem{BRZ}
J.~Borges, J.~Rifà, and V.A. Zinoviev, \emph{Families of completely transitive
  codes and distance transitive graphs}, Discrete Mathematics \textbf{324}
  (2014), 68 -- 71.

\bibitem{nested}
Joaquim Borges, Josep Rifà, and Victor~A. Zinoviev, \emph{Families of nested
  completely regular codes and distance-regular graphs}, Advances in
  Mathematics of Communications \textbf{9} (2015), no.~2, 233--246.

\bibitem{distreg}
A.~E. Brouwer, A.~M. Cohen, and A.~Neumaier, \emph{Distance-regular graphs},
  Ergebnisse der Mathematik und ihrer Grenzgebiete (3) [Results in Mathematics
  and Related Areas (3)], vol.~18, Springer-Verlag, Berlin, 1989.

\bibitem{Cam88}
Peter~J. Cameron, \emph{Geometric sets of permutations}, Geom. Dedicata
  \textbf{25} (1988), no.~1-3, 339--351, Geometries and groups
  (Noordwijkerhout, 1986).

\bibitem{Co1}
J.~H. Conway, \emph{{$M_{13}$}}, Surveys in combinatorics, 1997 ({L}ondon),
  London Math. Soc. Lecture Note Ser., vol. 241, Cambridge Univ. Press,
  Cambridge, 1997, pp.~1--11.

\bibitem{CEM}
J.~H. Conway, N.~D. Elkies, and J.~L. Martin, \emph{The {M}athieu group
  ${M}_{12}$ and its pseudogroup extension ${M}_{13}$}, Experiment. Math. 15
  \textbf{2} (2006), 223--236.

\bibitem{delsarte}
P.~Delsarte, \emph{An algebraic approach to the association schemes of coding
  theory}, Philips Res. Rep. Suppl. (1973), no.~10, vi+97.

\bibitem{permutation}
J.~D. Dixon and B.~Mortimer, \emph{Permutation groups}, Graduate Texts in
  Mathematics, vol. 163, Springer-Verlag, 1996.

\bibitem{fischer}
Bernd {Fischer}, \emph{{Finite groups generated by 3-transpositions. I}},
  {Invent. Math.} \textbf{13} (1971), 232--246.

\bibitem{GGNS}
N.~Gill, N.~I. Gillespie, A.~Nixon, and J.~Semeraro, \emph{Generating groups
  with hypergraphs}, 2015, To appear in {\it Quarterly J. Math.}

\bibitem{GGPS2}
N.~Gill, N.~I. Gillespie, C.~E. Praeger, and J.~Semeraro, \emph{Conway
  groupoids from triple systems and other structures}, 2016, in preparation.

\bibitem{GGPS}
\bysame, \emph{Conway groupoids, regular two-graphs and supersimple designs},
  2015, \url{http://arxiv.org/abs/1510.06680}.

\bibitem{GGS}
N.~Gill, N.~I. Gillespie, and J.~Semeraro, \emph{Conway groupoids and
  completely transitive codes}, 2015, To appear in {\it Combinatorica}.

\bibitem{giudici}
Michael Giudici and Cheryl~E. Praeger, \emph{Completely transitive codes in
  hamming graphs}, European Journal of Combinatorics \textbf{20} (1999), no.~7,
  647 -- 662.

\bibitem{HM}
D.~G. Higman and J.~E. McLaughlin, \emph{Geometric {$ABA$}-groups}, Illinois J.
  Math. \textbf{5} (1961), 382--397.

\bibitem{LS}
Martin~W. {Liebeck} and Jan {Saxl}, \emph{{Minimal degrees of primitive
  permutation groups, with an application to monodromy groups of covers of
  Riemann surfaces}}, {Proc. Lond. Math. Soc. (3)} \textbf{63} (1991), no.~2,
  266--314.

\bibitem{manning}
W.~A. Manning, \emph{The primitive groups of class $2p$ which contain a
  substitution of order $p$ and degree $2p$}, Trans. Amer. Math. Soc. 4
  \textbf{3} (1903), 351--357.

\bibitem{manning1}
\bysame, \emph{On the primitive groups of classes six and eight}, Amer. J.
  Math. \textbf{3} (1910), 235--256.

\bibitem{MS}
William~J. {Martin} and Bruce~E. {Sagan}, \emph{{A new notion of transitivity
  for groups and sets of permutations}}, {J. Lond. Math. Soc., II. Ser.}
  \textbf{73} (2006), no.~1, 1--13.

\bibitem{Nak}
Yasuhiro {Nakashima}, \emph{{The transitivity of Conway's $M_{13}$}}, {Discrete
  Math.} \textbf{308} (2008), no.~11, 2273--2276.

\bibitem{neum}
A.~Neumaier, \emph{Completely regular codes}, Discrete Math. \textbf{106/107}
  (1992), 353--360, A collection of contributions in honour of Jack van Lint.

\bibitem{binctrarb}
J.~Rif{\`a} and V.~A. Zinoviev, \emph{On a class of binary linear completely
  transitive codes with arbitrary covering radius}, Discrete Math. \textbf{309}
  (2009), no.~16, 5011--5016.

\bibitem{kronprod}
J.~Rif{\`a} and V.~A. Zinoviev, \emph{New completely regular {$q$}-ary codes
  based on {K}ronecker products}, IEEE Trans. Inform. Theory \textbf{56}
  (2010), no.~1, 266--272.

\bibitem{lifting}
Josep Rif{\`a} and Victor~A. Zinoviev, \emph{On lifting perfect codes}, IEEE
  Trans. Inform. Theory \textbf{57} (2011), no.~9, 5918--5925.

\bibitem{Seidel}
J.~J. Seidel, \emph{{On two-graphs and Shult's characterization of symplectic
  and orthogonal geometries over GF (2)}}, {T.H.-Report 73-WSK-02. Eindhoven,
  Netherlands: Technological University, Dept. of Mathematics. 25 p. (1973).}

\bibitem{Sh}
E.~Shult, \emph{Characterizations of certain classes of graphs}, J.
  Combinatorial Theory Ser. B \textbf{13} (1972), 142--167.

\bibitem{sole}
Patrick Sole, \emph{Completely regular codes and completely transitive codes},
  Discrete Mathematics \textbf{81} (1990), no.~2, 193 -- 201.

\bibitem{vantilborg}
Henricus Carolus~Adrianus van Tilborg, \emph{Uniformly packed codes},
  Technische Hogeschool Eindhoven, Eindhoven, 1976, With a Dutch summary,
  Doctoral dissertation, University of Technology Eindhoven. \MR{0414226}

\end{thebibliography}

\def\cprime{$'$}
\providecommand{\bysame}{\leavevmode\hbox to3em{\hrulefill}\thinspace}
\providecommand{\MR}{\relax\ifhmode\unskip\space\fi MR }
\providecommand{\MRhref}[2]{%
  \href{http://www.ams.org/mathscinet-getitem?mr=#1}{#2}
}
\providecommand{\href}[2]{#2}

\end{document}